\documentclass{amsart}

\usepackage{ae}
\usepackage[all]{xy}
\usepackage{graphicx,amsfonts,amssymb,amsmath}
\usepackage{natbib}

\usepackage{mathtools}

\usepackage[latin1]{inputenc}
\usepackage[T1]{fontenc}
\usepackage[francais]{babel}
\usepackage[cyr]{aeguill}

 \usepackage{hyperref}
\hypersetup{colorlinks,linkcolor=blue,citecolor=red,}

\usepackage[OT2,T1]{fontenc}
\DeclareSymbolFont{cyrletters}{OT2}{wncyr}{m}{n}
\DeclareMathSymbol{\sha}{\mathalpha}{cyrletters}{"58}
\input cyracc.def

 \newtheorem{thm}{Théorème}[section]
 \newtheorem*{thm*}{Théorème}

 \newtheorem{lem}[thm]{Lemme}
 \newtheorem{prop}[thm]{Proposition}
 \theoremstyle{definition}
 
 \theoremstyle{remark}
 
 \theoremstyle{remark}
 \newtheorem{rem}[thm]{\textbf{Remarque}}
 \numberwithin{equation}{subsection}
 \newcommand{\To}{\longrightarrow}

 \newcommand{\im}{\textup{Im}}

 \newcommand{\Br}{\textup{Br}}

  \newcommand{\Spec}{\textup{Spec}}
    \newcommand{\inv}{\textup{inv}}
        \newcommand{\pr}{\textup{pr}}
  \newcommand{\A}{\textbf{A}}

   \newcommand{\F}{\mathbb{F}}

 \newcommand{\E}{\mathcal{E}}

 \newcommand{\BM}{Brauer\textendash Manin}
 \newcommand{\oBM}{obstruction de Brauer\textendash Manin}

 \newcommand{\Q}{\mathbb{Q}}
 \newcommand{\Z}{\mathbb{Z}}
 
 \renewcommand{\O}{\mathcal{O}}
\begin{document}

\title[]
{Approximation forte sur un produit de variétés abéliennes épointé en des points de torsion}

\author{ Yongqi Liang}

\address{Yongqi LIANG
\newline University of Scinece and Technology of China,
\newline School of Mathematical Sciences,
\newline 96 Jinzhai Road,
\newline  230026 Hefei, Anhui, China
 }

\email{yqliang@ustc.edu.cn}

%\thanks{This work was completed with the help of Prof.  .}

\thanks{\textit{Mots clés} : approximation forte, obstruction de Brauer\textendash Manin, variétés abéliennes, points de torsion, points entiers}


%%% ----------------------------------------------------------------------

%%% ----------------------------------------------------------------------
\maketitle

%\tableofcontents

%%% ----------------------------------------------------------------------

\begin{abstract}
Considérons l'approximation forte pour les variétés algébriques définies sur un corps de nombres $k$.
Soit  $S$ un ensemble fini de places de $k$ contenant les places archimédiennes.
Soit $E$ une courbe elliptique de rang de Mordell\textendash Weil non nul et soit $A$ une variété abélienne de dimension strictement positive et de groupe de Mordell\textendash Weil fini. Pour  un ensemble fini quelconque $\mathfrak{T}$ de points de torsion de $E\times A$,  notons par $X$ son complémentaire. En supposant la finitude de $\sha(E\times A)$, nous démontrons que $X$ vérifie l'approximation forte avec l'\oBM\ hors de $S$ si et seulement si la projection de $\mathfrak{T}$ sur $A$ ne contient aucun point $k$-rationnel. 
\end{abstract}

\begin{center}\footnotesize\textbf{STRONG APPROXIMATION FOR PRODUCTS OF ABELIAN VARIETIES\ \\  PUNCTURED AT TORSION POINTS}
\end{center}
\renewcommand{\abstractname}{Abstract}
\begin{abstract}
Consider strong approximation for algebraic varieties defined over a number field $k$.
Let $S$ be a finite set of places of $k$ containing all archimedean places. Let $E$ be an elliptic curve  of positive Mordell\textendash Weil rank and let $A$ be an abelian variety of positive dimension and of finite Mordell\textendash Weil group. For an arbitrary finite set  $\mathfrak{T}$ of torsion points of $E\times A$, denote by $X$ its complement. Supposing the finiteness of $\sha(E\times A)$, we prove that $X$ satisfies strong approximation with Brauer\textendash Manin obstruction off $S$ if and only if the projection of $\mathfrak{T}$ to $A$ contains no $k$-rational points.
\end{abstract}

\normalsize

\section{Introduction}\label{intro}
Soit $k$ un corps de nombres. On s'intéresse à étudier les points entiers d'une $k$-variété lisse. On note par $X^{c}$ une compactification lisse de $X$. En général, même si on connaît le comportement des points rationnels de $X^{c}$,  on ne peut qu'en déduire très peu d'informations sur les points entiers de $X$.

Depuis très longtemps,  on s'intéresse à la densité de Zariski de l'ensemble des points entiers. On s'intéresse également à l'approximation forte \textemdash \ la densité de l'ensemble des points entiers plongé  dans l'ensemble des points adéliques. Dans la littérature, de variétés de types divers ont été  étudiées.  Dans cet article, on se limite aux ouverts de certaines variétés abéliennes.

Dans l'article de B. Hassett et Y. Tschinkel \cite{HassettTschinkel}, ils ont discuté la question de  la densité potentielle de Zariski des points entiers.  Ils ont démontré que tout ouvert d'un produit de variétés abéliennes simples dont le complémentaire a grand codimension vérifie cette densité potentielle. A. Kresch et Y. Tschinkel \cite{KreschTschinkel} ont aussi les résultats numériques qui soutiennent une réponse affirmative à cette question pour certaines jacobiennes épointées en un point rationnel.

D'un autre point de vue, dans l'article récent \cite{CLX} de Y. Cao, F. Xu et l'auteur, nous avons étudié la densité des points entiers au sens de la topologie adélique. Nous avons répondu partiellement à une question de O. Wittenberg, cf. \cite[Problem 6]{AIM2014} et \cite[Question 2.11]{WittSurvey}.
Parmi nos résultats, nous avons montré  le théorème suivant.
\begin{thm*}[{\cite[Corollary 8.2]{CLX}}]
Soit $E$ une $k$-courbe elliptique de rang de Mordell\textendash Weil non nul. Soit $A$ une $k$-variété abélienne de dimension strictement positive et de rang de Mordell\textendash Weil zéro. 

Si $O$ désigne l'élément neutre de $E\times A$, alors  l'ouvert  $(E\times A)\setminus O$ ne vérifie pas  l'approximation forte avec l'\oBM\  hors de $\infty$.
\end{thm*}
Sa démonstration est basée sur une généralisation d'un argument de D. Harari et J. F. Voloch \cite{HV10} et sur une idée qui remonte à B. Poonen \cite{Poonen}.

Le but de cet article est de généraliser ce théorème. Nous gardons les même hypothèses sur $E$ et $A$,  considérons  le complémentaire $X$ d'un ensemble quelconque de points de torsion de $E\times A$. En améliorant l'ancien argument, nous trouvons une description complète pour la propriété d'approximation forte sur $X$.

\begin{thm}\label{thmIntro}
Soit $k$ un corps de nombres.  Soit $S$ un ensemble fini de places de $k$ contenant les places archimédiennes.

Soit $E$ une $k$-courbe elliptique de rang de Mordell\textendash Weil non nul. Soit $A$ une $k$-variété abélienne de dimension strictement positive et de rang de Mordell\textendash Weil zéro.  Pour $\mathfrak{T}$ un ensemble quelconque de points de torsion de $E\times A$, on note par $X$ son complémentaire.

Si la projection de $\mathfrak{T}$ sur $A$ contient un point $k$-rationnel, alors $X$ ne vérifie pas  l'approximation forte avec l'\oBM\  hors de $S$.
En supposant la finitude du groupe de Tate\textendash Shafarevich $\sha(E\times A)$, si la projection $\mathfrak{T}$ sur $A$ ne contient pas de points $k$-rationnels, alors $X$ vérifie  l'approximation forte avec l'\oBM\  hors de $S$.
\end{thm}

Dans \S \ref{enonce}, nous rappelons le contexte concerné et présentons  un énoncé plus précis, dont le Théorème \ref{thmIntro} est une consequence. Dans \S \ref{demonstration}, nous  détaillons sa démonstration.

\section{Énoncé précis du résultat}\label{enonce}
Dans cet article,  le corps de base $k$ est toujours un corps de nombre quelconque.
Pour une place $v$ appartient à   l'ensemble $\Omega$ des places de $k$, on note par $k_{v}$ le complété de $k$. On note par $\infty\subset\Omega$ l'ensemble des places archimédiennes. Pour $v\in\Omega\setminus\infty$, on note par $\O_{k_{v}}$  l'anneau des entiers de $k_{v}$. On note par $\A$ l'anneau des adèles de $k$, et par $\A^{S}$ l'anneau des $S$-adèles si $S\subset\Omega$ est un sous-ensemble fini de places.

Soit $X$ une $k$-variété (schéma séparé de type fini, géométriquement intègre sur $k$) lisse.  L'accouplement de \BM \ est défini par
$$X(\A)\times \Br(X)\To\Q/\Z$$
$$(x_{v})_{v\in\Omega},b\mapsto\sum_{v\in\Omega}\inv_{v}(b(x_{v})),$$
où $\inv_{v}:\Br(k_{v})\xhookrightarrow{}\Q/\Z$ désigne l'invariant local en $v$ provenant de la théorie des corps de classes locaux.
Pour un sous-ensemble $B\subset \Br{X}$, le sous-ensemble $X(\A)^{B}$ des familles de points locaux qui sont orthogonales à tous les éléments appartenant à  $B$ est fermé dans l'espace des points adéliques $X(\A)$. Il contient l'ensemble des points rationnels $X(k)$ d'après la théorie des corps de classes globaux, et il contient ainsi son adhérence $\overline{X(k)}$. 
 
Rappelons que $X$ vérifie \emph{l'approximation forte avec l'\oBM\ par rapport à $B$ hors de $S$} si $X(k)\subset X(\A^{S})$ est dense dans $\pr^S(X(\A)^{B})\subset X(\A^{S})$, où $\pr^S:X(\A)\to X(\A^{S})$ est la projection en oubliant les composantes des places dans $S$. Si $B=\Br(X)$, on dit simplement que $X$ vérifie \emph{l'approximation forte avec l'\oBM\  hors de $S$}.

Dans l'introduction, nous avons énoncé le cas particulier où $A$ et $E$ sont des variétés abéliennes du Corollaire 8.2 de \cite{CLX}. En effet, ce dernier corollaire était énoncé et démontré dans \cite{CLX} pour les variétés semi-abéliennes. Dans le reste de cet article,  nous démontrons  le théorème suivant qui généralise \cite[Corollary 8.2]{CLX} au sens que le fermé $F$ n'est pas forcément de dimension $0$ et que $F$ ne contient même pas nécessairement un point $k$-rationnel.
  
\begin{thm}\label{mainThm}
Soit $k$ un corps de nombres.  Soit $S\supset\infty$ un ensemble fini de places de $k$ contenant les places archimédiennes.
Soient $A$ et $E$ des variétés semi-abéliennes définies sur $k$.
Supposons que  $A$ est de dimension strictement positive et  que $A(k)$ est discret dans $A(\A^{S})$. 
Supposons que $E$ est de dimension $1$ et que $E(k)$ n'est pas discret dans $E(\A^{\infty})$.

Soit $F\subset E\times A$ un fermé de codimension $\geq2$. Supposons que la projection de $F$ sur $A$ contient au moins un point $k$-rationnel $P$ tel que la fibre $F\times_{A}P$, considérée  comme un fermé de Zariski de $E$, ne consiste qu'en  points de torsion (pas forcément $k$-rationnels) de $E$. 

Alors $X=(E\times A)\setminus F$ ne vérifie pas l'approximation forte avec l'\oBM\ hors de $S$.
\end{thm}

Lorsque $E$ est une courbe elliptique, sur laquelle l'hypothèse est équivalente à la condition que  son rang de Mordell\textendash Weil est non nul. La première conclusion du Théorème \ref{thmIntro} en découle directement. Expliquons la seconde conclusion comme suit.

\begin{rem} 
En supposant la finitude des groupes de Tate\textendash Shafarevich $\sha(A^{\textup{ab}})$ du quotient abélien $A^{\textup{ab}}$ de $A$ et  $\sha(E)$ quand $E$ est une courbe elliptique, l'hypothèse partielle du théorème \ref{mainThm} que la projection de $F$ sur $A$ contient au moins un point $k$-rationnel  est nécessaire. Si la projection de $F$ ne contient aucun point $k$-rationnel, alors $X$ vérifie l'approximation forte avec l'\oBM\ hors de $S$. En effet, il suffit de considérer l'hypothèse la plus faible et la conclusion la plus forte, disons $S=\infty$. Ceci résulte d'un argument  simple de fibration:
si $(x_{v})_{v\in\Omega}$ est orthogonale à $\Br(X)$, sa projection sur $A$ est alors orthogonale à $\Br(A)$, cela entraîne que cette projection (en oubliant les composantes archimédiennes) provient d'un point $k$-rationnel $P$ d'après un résultat de D. Harari \cite[Théorème 4]{Harari08}
car $A(k)$ est supposé discret dans $A(\A^{\infty})$. Comme la projection de $F$ ne contient aucun $k$-point, la fibre $X\times_{A}P$ est identique à $E$. Elle contient donc $(x_{v})_{v\in\Omega}\bot\Br(E)$ quitte à modifier les composantes archimédiennes si nécessaire. La famille de points locaux $(x_{v})_{v\in\Omega\setminus\infty}$ peut être approximée  par un point global d'après \cite[Théorème 4]{Harari08}.

\end{rem}

\section{Démonstrations du énoncé}\label{demonstration}

On démontre le Théorème \ref{mainThm} dans cette section. Tout d'abord, on compare certains groupes de Brauer dans le Lemme \ref{gpBr}. Le Théorème \ref{mainThm} est une conséquence directe des Lemmes \ref{gpBr}, \ref{fibration} et de la Proposition  \ref{ellipticcurve}. À la fin, on démontre la Proposition \ref{ellipticcurve}. 

\begin{lem}\label{gpBr}
Soient $A,$ $E,$ $X,$ $F,$ et $P$ comme dans le théorème \ref{mainThm}. Le point $P\in A(k)$ induit une immersion fermée $i_{P}:X\times_{A}P\to X$. Si on identifie $X\times_{A}P$  avec l'ouvert $E\setminus (F\times_{A}P)$ de $E$, alors $$\im[i_{P}^{*}:\Br(X)\to\Br(X\times_{A}P)]=\im[\Br(E)\to\Br(X\times_{A}P)].$$
\end{lem}

\begin{proof}
Considérons le diagramme commutatif suivant
\SelectTips{eu}{12}$$\xymatrix@C=20pt @R=14pt{
\Br_{}(E\times A)\ar[r]\ar[d]&\Br_{}(E)\ar[d]
\\  \Br(X)\ar[r]^-{i_{P}^{*}}&\Br_{}(X\times_{A}P),
}$$ 
où les flèches verticales sont induites par les immersions ouvertes et 
où la flèche horizontale en haut est induite par la section associée au point rationnel $P\in A(k)$. Comme $F=(E\times A)\setminus X$ est de codimension $\geq2$, la flèche à gauche est un isomorphisme d'après le théorème de pureté pour les groupes de Brauer. L'égalité voulue résulte alors de la surjectivité de la flèche en haut.
\end{proof}

\begin{lem}\label{fibration}
Soit $k$ un corps de nombres. Soit $S$ un ensemble fini de places de $k$.
Soit $f:X\to A$ un morphisme entre de variétés algébriques définies sur un corps de nombres $k$. 
Supposons que $A(k)\subset A(\A^{S})$ est discret.
Soit $P\in A(k)$ un point rationnel de $A$ tel que la fibre $X_{P}=X\times_{A}P\neq\emptyset$.

Si $X$ vérifie l'approximation forte avec \oBM\ hors de $S$, alors $X_{P}$ vérifie l'approximation forte avec \oBM\ par rapport à $i_{P}^{*}(\Br(X))$ hors de $S$.
\end{lem}

\begin{proof} C'est un argument standard de fibration.
Il existe un ouvert $U$ de $A(\A^{S})$ tel que $U\cap A(k)=\{P\}$. 
Soit $V_{P}\subset X_{P}(\A^{S})$ un ouvert tel que 
$$[V_{P}\times\prod_{v\in S}X_{P}(k_{v})]\cap X_{P}(\A)^{i^{*}_{P}(\Br(X))}\neq\emptyset.$$
L'ouvert $V_{P}$ est alors la restriction d'un ouvert $V\subset f^{-1}(U)\subset X(\A^{S})$ à la fibre $X_{P}$. D'après la fonctorialité de l'accouplement de \BM, 
$$[V\times\prod_{v\in S}X(k_{v})]\cap X(\A)^{\Br(X)}\neq\emptyset.$$
Comme $X$ vérifie l'approximation forte avec \oBM\  hors de $S$, l'ensemble des points rationnels $X(k)$ intersecte avec $V$. 
Si $Q\in X(k)\cap V$, alors $f(Q)=P$ et $Q\in X_{P}(k)\cap V_{P}$. Donc $X_{P}$ vérifie l'approximation forte par rapport à $i_{P}^{*}(\Br(X))$ hors de $S$.
\end{proof}

Le  Théorème \ref{mainThm} résulte directement des lemmes précédents et de la proposition suivante. 

\begin{prop}\label{ellipticcurve}
Soit $k$ un corps de nombres.  Soit $S\supset\infty$ un ensemble fini de places de $k$ contenant les places archimédiennes.
Soit $E$ une variété semi-abélienne de dimension $1$ définie sur un corps de nombres $k$. Supposons que $E(k)$ n'est pas discret dans $E(\A^{\infty})$.
Soit $\mathfrak{M}=\{m_{1}, m_{2}, \ldots, m_{s}\}\subset E$ un sous-ensemble fini non-vide de points de torsion de $E$. 

Si on note par $E_{0}$ le complémentaire de $\mathfrak{M}$ dans $E$, alors $E_{0}$ ne vérifie pas l'approximation forte avec l'\oBM\ par rapport à $\Br(E)$ hors de $S$.
\end{prop}

\begin{proof} La preuve généralise l'argument de Harari et Voloch \cite{HV10} sur la courbe elliptique $y^{2}=x^{3}+3$ définie sur $\Q$. Une partie   de l'argument suivant a été établie dans \cite[Lemma 6.6, Theorem 8.1]{CLX} pour traiter le cas où $\mathfrak{M}$   consiste en un seul point rationnel et pour l'énoncé sur un corps de nombres $k$ quelconque.  La nouveauté ici est que $\mathfrak{M}$ peut contient plusieurs de points fermés non nécessairement rationnels.

D'abord, on démontre la proposition pour le cas où $E$ est une courbe elliptique. Lorsque $E$ est un tore de dimension $1$, la preuve est essentiellement la même. On explique à la fin les modifications nécessaires pour le cas de tores.

À partir de maintenant, supposons que $E$ est une courbe elliptique. Soit $\E$ son modèle de Néron sur l'anneau des entiers $\O_{k}$.
Pour tout $j\in\mathbb{N}$, $1\leq j\leq s$, soit $M_{j}$ l'adhérence de Zariski de $m_{j}$ dans $\E$. La normalisation de $M_{j}$ est $\Spec(\O_{k_{j}})$, où $\O_{k_{j}}$ est l'anneau des entiers du corps résiduel $k_{j}$ de $m_{j}$. Soit $\E_{0}$ le complémentaire dans $\E$ de la réunion de $M_{j}$. Soit $n\in\mathbb{N}$ tel que $nm_{j}=O\in E(k_{j})$ pour tout $j$.

D'après l'hypothèse, $E$ est de rang de Mordell\textendash Weil strictement positif. Fixons un point rationnel $Q\in E(k)$ d'ordre infini. Il définit une section de $\E$ notée encore par $Q\in\E(\O_{k})$.
Les sous-ensembles suivants de $\Omega$  sont alors finis.

\begin{equation*}
    \begin{array}{rcl}
       T_{1} & = &  \{ v\in \Omega\setminus \infty :   Q\text{ intersecte avec }M_{j}\text{ au-dessus de } v\text{ pour un certain }j\} \\
        T_{2}          & = & \{ v\in \Omega\setminus \infty : \Spec(\O_{k_{j}})\to\Spec(\O_{k}) \text{ est ramifié au-dessus} \\ 
        && \hspace{198pt} \text{de }v\text{ pour un certain }j\}\\
        T_{3} & = & \{ v\in \Omega\setminus \infty : \Spec(\O_{k_{j}})\to M_{j}\text{ admet une fibre non triviale} \\ 
        && \hspace{154pt}\text{au-dessus de }v\text{ pour un certain }j\}\\
            T & = & T_{1}\cup T_{2}\cup T_{3}\cup S
    \end{array}
\end{equation*}
D'après le théorème de Seigel, l'ensemble des points $T$-entiers $$\E_{0}(\O_T)=\{P_1=Q,P_2,\cdots, P_t\}$$ est  fini.

Choisissons une place non-dyadique $v_{0}\in\Omega\setminus T$ telle que $Q\not\equiv P_i \mod v_0$ pour $2\leq i\leq t$ et telle que $E$ admet une bonne réduction en $v_{0}$. Soit $q$ la caractéristique du corps résiduel $\F_{v_{0}}$ de $\O_{k}$ en $v_{0}$. On trouve 
 $$\text{soit}\ \ {\rm pgcd} (n|\E(\F_{v_0})|+1, q)=1, \ \ \ \text{soit} \ \ \  {\rm pgcd} (n|\E(\F_{v_0})|-1, q)=1 . $$  
Définissons 
 $$ a_{}=\begin{cases} n|\E(\F_{v_0})|+1, \ \ \ & \text{si ${\rm pgcd}(n|\E(\F_{v_0})|+1, q)=1$;} \\
(q-1)n|\E(\F_{v_0})|+1, \ \ \ & \text{si ${\rm pgcd}(n|\E(\F_{v_0})|-1, q)=1$.} \end{cases} $$
Alors $a$ est toujours premier avec $qn|\E(\F_{v_0})|$. D'après le théorème de la progression arithmétique  de Dirichlet,
soit $\Lambda$ l'ensemble infini des nombres premiers $l$ qui vérifient 
$$l \equiv a_{}  \mod qn|\E(\F_{v_0})| .$$
Donc $n|\E(\F_{v_0})|$ divise $l-1$ et la valuation $\textup{val}_q(l-1)=\textup{val}_q(n|\E(\F_{v_0})|)$ est une constante pour tout $l\in \Lambda$.

Pour $v\in \infty$, désignons $\pi_0(E(k_v))$ le groupe des composantes connexes du groupe de Lie  $E(k_v)$.
Considérons la suite $( lQ )_{l\in\Lambda}$ plongée dans l'espace compact 
$$\prod_{v\in \infty} \pi_0(E(k_v)) \times \prod_{v\notin\infty}\E(\O_{k_{v}}). $$ Il existe alors une sous-suite convergente vers un élément noté par   $(x_v)_{v\in\Omega}$ qui est ainsi orthogonal  à $\Br(E)$ d'après la continuité de l'accouplement de \BM.

Pour toute $v\in\infty$, le composé $E_0(k_v)\to E(k_v)\to\pi_{0}(E(k_{v}))$ est surjectif. Quitte à modifier $x_v$ si nécessaire, on peut supposer que $x_v\in E_0(k_v)$ pour toute place archimédienne sans affecter l'orthogonalité avec le groupe de Brauer.

Pour toute  $v\notin \infty$, démontrons que $x_v \in E_0(k_v)$. En effet, il suffit de considérer le cas où il existe une place $w$ de $k_{j}$ au-dessus de $v$ telle que l'extension $k_{j,w}/k_{v}$ est triviale, sinon le corps résiduel de chaque  point du support du 0-cycle  $m_{j,k_{v}}=m_{j}\times_{\Spec(k)}\Spec(k_{v})$ contient strictement  $k_{v}$ ainsi que $x_v$ n'est jamais contenu dans $m_{j,k_{v}}$. Soit $w$ une telle place et soit  $m_{j,w}$ l'image de $m_{j}$  par $E(k_{j})\to E(k_{j,w})$.  Comme $n$ divise $l-1$, on trouve que $lm_{j,w}=m_{j,w}$ dans $E(k_{j,w})$, ou bien $lM_{j,w}=M_{j,w}$ dans $\E(\O_{k_{j,w}})$ où $M_{j,w}=M_{j}\times_{\Spec(\O_{k})}\Spec(\O_{k_{j,w}})$.
Puisque $Q$ n'est pas de torsion $Q\neq m_{j,w}\in E(k_{j,w})$,  il existe un entier $r$ strictement positif tel que les réductions de $Q\in\E(\O_{k})$ et de $M_{j,w}\in\E(\O_{k_{j,w}})$ sont différentes dans $\E(\O_{k_{j,w}}/(\pi_{w}^{r}))$ où $\pi_{w}$  est une uniformisante de $k_{j,w}$. Si la limite $x_v$ de $lQ$ est égale à $m_{j,w}$, il existe alors un nombre infini de premiers $l$ tel que $lQ$ co\"incide avec $M_{j,w}=lM_{j,w}$ dans $\E(\O_{k_{j,w}}/(\pi_{w}^{r}))$. Cela contredit au fait que les réductions de $Q$  et de $M_{j,w}$ sont différentes  dans $\E(\O_{k_{j,w}}/(\pi_{w}^{r}))$.

Pour toute $v\not \in T$, démontrons que $x_v\in\E_{0}(\O_{k_{v}})$, autrement dit la réduction  $\bar{x}_{v}$ de $x_v$ modulo $v$ ne se trouve pas dans  $\bigcup_{j=1}^{s}M_{j}$. D'après la choix de $T$, la réduction $M_{j}\times_{\Spec(\O_{k})}\Spec(\F_{v})$ est une sous-schéma fermé réduit de $\E_{v}=\E\times_{\Spec(\O_{k})}\Spec(\F_{v})$. Si elle contient $\bar{x}_{v}$, il existe alors une place $w$ de $k_{j}$ non-ramifiée et de degré $1$ au-dessus de $v$. La réduction $\bar{M}_{j,w}\in\E_{v}(\F_{w})$ de $M_{j,w}\mod w$ co\"incide  avec $\bar{x}_{v}$.  
Comme $n$ divise  $l-1$ et l'ordre de $m_{i}$ divise $n$, on trouve que   $m_{j}=lm_{j}\in E(k_{j})$, d'où $\bar{M}_{j,w}=l\bar{M}_{j,w}\in\E_{v}(\F_{w})$. Pour $l$ suffisamment grand, la réduction $l\bar{Q}_{v}$ de $lQ\mod v$  est égale à $\bar{x}_{v}=\bar{M}_{j,w}=l\bar{M}_{j,w}\in\E_{v}(\F_{w})$. Ceci entraîne que  $l\bar{Q}_{v}=l\bar{M}_{j,w}\in\E_{v}(\F_{w})$ pour un nombre infini de premiers $l$, d'où  $\bar{Q}_{v}=\bar{M}_{j,w}\in\E_{v}(\F_{w})$ qui contredit le fait que $Q$ et $M_{j}$ n'intersectent pas en dehors de $T$. Par conséquent $x_v\in\E_{0}(\O_{k_{v}})$ pour toute $v\notin T$.

Nous concluons que 
$$(x_v)_{v\in \Omega}\in [(\prod_{v\in T}E_{0}(k_{v}))\times(\prod_{v\notin T}\E_{0}(\O_{k_{v}}))]^{\Br(E)} . $$ 

Afin de compléter la preuve, il reste à démontrer que  pour tout $S\supset\infty$
la troncation
$$(x_v)_{v\not\in S}\in \textup{pr}^{S} \left([(\prod_{v\in T}E_{0}(k_{v}))\times(\prod_{v\notin T}\E_{0}(\O_{k_{v}}))]^{\Br(E)}\right) $$ 
ne se trouve pas  dans l'adhérence de $E_0(k)\subset E_0(\A^S)$.
Comme 
$$\displaystyle E_0(k)\cap[(\prod_{v\in T\setminus S}E_0(k_v))\times(\prod_{v\notin T}\E_0(\O_{k_v}))]=\E_0(\O_T)=\{P_1=Q,P_2,\cdots, P_t\}$$ est fini, $E_0(k)$ est discret dans $E_0(\A^S)$. 
Il suffit de démontrer que $(x_v)_{v\notin S}$ n'est l'image d'aucun des points $P_i$.
Observons que $l Q\equiv Q  \mod v_0$ car $|\E(\F_{v_0})|$ divise $l-1$.
Pour $2\leq i\leq t$, on trouve que $Q \not\equiv P_i \mod v_0$ d'après la choix de $v_{0}$, et ainsi que $x_{v_{0}} \not\equiv P_i \mod v_0$. Donc  $(x_v)_{v\notin S}$  n'est pas l'image de $P_i$ pour $2\leq i\leq t$.
Par l'absurde, supposons que $(x_v)_{v\notin S}=P_1=Q$. Alors une sous-suite de $(lQ)_{l\in\Lambda}$ converge vers $Q$ dans $E(k_{v_{0}})$,  disons $(l-1)Q\to O$ pour les premiers $l$ apparus comme indices de la sous-suite convergente. De l'autre côté,  $E(k_{v_0})$ contient un sous-groupe d'indice finie qui est isomorphe à $(\O_{k_{v_0}},+)$ comme groupes topologiques d'après  \cite[Theorem 7]{Mattuck}. Comme $Q$ n'est pas un point de torsion, si on note cette dernière indice par $N$, l'élément $Q'=NQ\in E(k_{v_{0}})$ est alors non nul. En plus, $Q'$ se trouve dans le sous-groupe qui peut être identifié avec $(\O_{k_{v_{0}}},+)$. La convergence $(l-1)Q'=(l-1)NQ\to O$ implique que $l\to 1$ dans $\O_{k_{v_0}}$ qui contredit le fait que la valuation $\textup{val}_q(l-1)=\textup{val}_q(n|\E(\F_{v_0})|)$ est une constante pour tout $l\in\Lambda$.

Enfin,  expliquons l'adaptation nécessaire lorsque $E$ est un tore de dimension $1$ au lieu d'une courbe elliptique.
Dans la preuve, nous avons besoin d'un $\O_k$-modèle $\E$ qui est un schéma en groupe de type fini et lisse sur $\O_k$, pour un tore nous prenons la composante connexe de l'identité du lft-modèle de Néron, cf. \cite[Theorem 5.12]{Neronmodel}. 
D'après le théorème de Dirichlet généralisé  \cite[Theorem 5.12]{AlgGpNT}, le groupe des unités $\E(\O_k)$ est un groupe abélien de type fini. Il est de rang strictement positif car $E(k)\subset E(\A^\infty)$ n'est pas discret. Nous fixons  $Q\in\E(\O_k)$ un point d'ordre infini.
 Un point de torsion $m_j\in\mathfrak{M}$ s'étend en une section $M_j\in\E(\O_k)$. Tout  reste de la preuve fonctionne dans ce contexte.
\end{proof}

\footnotesize
%\noindent\textbf{Remerciements.}  L'auteur tient à remercier J.-L. Colliot-Thélène  pour ses suggestions.
%   funding 

\normalsize
%\nocite{*} %%show all the bib
\bibliographystyle{alpha}
\bibliography{mybib1}

\end{document}